\documentclass[letter,onecolumn,12pt]{IEEEtran2}

\usepackage{amssymb}
\usepackage{amsmath}
\usepackage{amsthm}
\usepackage{latexsym}
\usepackage{epsfig}
\usepackage{graphicx,subfigure}
\usepackage{color}
\usepackage{float}
\usepackage{cite}
\usepackage{enumerate}

\usepackage{indentfirst}
\usepackage{amsfonts}
\usepackage{amsmath}
\usepackage{amssymb}
\usepackage{amsthm}
\usepackage{mathrsfs}
\usepackage{graphicx,epsfig}
\usepackage{epsfig}
\long\def\comment#1{}

\renewcommand{\thesection}{\arabic{section}}
\renewcommand{\thesubsection}{\arabic{section}.\arabic{subsection}}

\newcommand{\myas}{\mbox{\rm a.s.}}

\newcommand{\myE}{\mathbb{E}}
\newcommand{\myF}{\mathscr{F}}
\newcommand{\myFb}{\mathbb{F}}
\newcommand{\myG}{\mathscr{G}}
\newcommand{\myH}{\mathcal {H}}
\newcommand{\myL}{\mathcal {L}}

\newcommand{\myLiz}{{l^2_\infty(\Omega,\myF,\myFb; \myR^{n_z})}}
\newcommand{\myLiv}{{l^2_\infty(\Omega,\myF,\myFb; \myR^{n_v})}}

\newcommand{\myN}{\mathbb {N}}
\newcommand{\myNz}{\mathcal {N}}
\newcommand{\myP}{\mathbb{P}}
\newcommand{\myR}{\mathbb{R}}
\newcommand{\myfunV}{{\mathbb{V}}}
\newcommand{\myBeta}{{\mathscr{B}_\myfunV}}
\newcommand{\myS}{\mathcal {S}}

\renewcommand\thesection{\arabic{section}}
\newtheorem{theorem}{Theorem}[section]
\newtheorem{corollary}{Corollary}[section]
\newtheorem{remark}{Remark}[section]
\newtheorem{definition}{Definition}[section]
\newtheorem{proposition}{Proposition}[section]
\newtheorem{lemma}{Lemma}[section]
\newtheorem{example}{Example}[section]

\begin{document}

\title{New Approach to General Nonlinear  Discrete-Time  Stochastic  $H_\infty$ Control
}

\author{Xiangyun Lin, ~Tianliang Zhang, ~Weihai Zhang$^\dag$,~
        ~and~ Bor-Sen Chen
\thanks{X. Lin is with the College of Mathematics and Systems Science, Shandong University of Science and
Technology, Qingdao, 266590, China.}
\thanks{T. Zhang is with the School  of Automation  Science  and
Engineering, South China University  of Technology, Guangzhou,
510641,  China.}
\thanks{W. Zhang is with the College of Electrical Engineering and Automation,
Shandong University of Science and Technology, Qingdao, 266590,
China.}
\thanks{B. S. Chen is with the Department of Electrical Engineering,
National Tsing Hua University, Hsinchu, 30013, Taiwan.}
\thanks{$\dag$ Corresponding author. Email: w\_hzhang@163.com.}
}


\maketitle

\begin{abstract}
In this paper, a  new  approach based on convex analysis  is
introduced  to solve the $H_\infty$ problem for discrete-time
nonlinear stochastic systems. A stochastic version of bounded real
lemma is proved and  the state feedback $H_\infty$ control  is
studied. Two examples are presented to show the effectiveness of our
developed theory.
\end{abstract}



\section{Introduction}

$H_\infty$ theory  was initially formulated by Zames
\cite{Zames1981} in the early 1980's for linear time-invariant
systems, where the $H_\infty$ norm, defined in the frequency-domain
form for a stable transfer matrix, plays an important role in robust
linear control design; see
 \cite{BasarBernhard1995}  and \cite{FrancisDoyle1987}. A
breakthrough of the classical $H_\infty$  theory in \cite{DGKF1989}
initiated    the  time-domain state-space approach in the $H_\infty$
study,  and turned  the  $H_\infty$ controller  design into  solving
two algebraic Riccati equations (AREs).  After the appearance of
\cite{DGKF1989}, $H_\infty$ control theory has made a great progress
in the 1990's \cite{ZhouDoyle98}. Up to now, $H_\infty$ control has
been successfully  applied to network control \cite{hudeng2017},
synthetic biology design \cite{wuzhangchen2011, Chenchanglee2009},
etc..

Instead of solving   two Riccati equations or  Riccati inequalities
as  in \cite{DGKF1989} , Gahinet and
Apkarian\cite{GahinetApkarian1994} introduced the linear matrix
inequality (LMI) approach to the $H_\infty$ controller design, which
is more convenient due to the usage of LMI Toolbox. In the
time-domain framework, the $H_\infty$ control theory is first
extended to   nonlinear deterministic systems expressed  by ordinary
differential equations(ODEs). For example, based on the solutions of
Hamilton-Jacobi equations or inequalities, the state feedback
$H_\infty$  control  \cite{VanderSchaft1991}  and output feedback
$H_\infty$ control \cite{BallHeltonWalker1993}, \cite{isidori_92},
 were discussed, respectively. The reference
\cite{HinrichsenPritchard1998} first systematically studied the
stochastic $H_\infty$ control of linear It\^o systems, where  a
stochastic  bounded real lemma was obtained in terms of linear
matrix inequalities (LMIs), and the dynamic output feedback
$H_\infty$ problem was also discussed. At the same time, the state
feedback $H_\infty$ control for linear time-invariant It\^o systems
with state-dependent noise was also discussed in \cite{ug1998} based
on stochastic differential game. We refer the reader to the
monograph \cite{pete}  for the early development in the $H_\infty$
control theory of linear  It\^o systems. Except for the $H_\infty$
estimation, the extended Kalman filtering on stochastic It\^o
systems was also discussed in \cite{Germani2007}.
 By means of completing the squares and stochastic dynamic programming, the state-feedback $H_\infty$ control and robust $H_\infty$
 filtering were  extensively investigated  in \cite{ZhangChen2006} and \cite{ZhangChenTseng2005} for affine  stochastic It\^o
 systems.  It can be founded that starting from 1998, the stochastic $H_\infty$ control has become a popular research field \cite{zhangxiechen2017}, which  has  been
 extended to other stochastic systems  such as Markovian
 jumps \cite{BoukasLiu2001,  XuLamMao2007, v.dragan_2005}, Poisson jumps \cite{Lin2011} and  L\'{e}vy processes \cite{ChenWu2015}.

With the development of  $H_\infty$ control theory of
continuous-time It\^o systems, the discrete-time $H_\infty$ control
has  also attracted considerable attention. For deterministic linear
systems, Basar and Bernhard \cite{BasarBernhard1995} have developed
the discrete-time counterpart  of the continuous-time $H_\infty$
design. Based on the dissipation inequality, differential game, and
LaSalle's invariance principle, Lin and Byrnes \cite{LinByrnes1996}
developed the $H_\infty$ control theory for general nonlinear
discrete-time deterministic  systems. Bouhtouri, Hinrichsen and
Pritchard \cite{BouhtouriHinrichsenPritchard1999} first studied  the
$H_\infty$-type control for  discrete-time linear stochastic systems
with multiplicative noise.  The infinite horizon mixed
$H_2/H_\infty$ control for discrete-time stochastic systems with
state and disturbance dependent noise can be found in
\cite{ZhangHangXie2008}, which turned  out that the mixed
$H_2/H_\infty$ controller design is associated with the solvability
of the four coupled matrix-valued equations. For the disturbance
attenuation problem of linear discrete-time multiplicative noise
systems with Markov jumps, we refer the reader to
\cite{dragan_2010}.  Berman and Shaked \cite{BermanShaked2006} first
explored the general discrete-time stochastic $H_\infty$ control
problem, and presented a bounded real lemma in terms Hamilton-Jacobi
inequality, where the Hamilton-Jacobi inequality contains the
supremum of some conditional mathematical expectation. As an
application, for a class of discrete-time time-varying nonlinear
stochastic systems with multiplicative noises,  a relatively easily
testing criterion was derived via taking the Lyapunov function to be
a quadratic form.  In \cite {luzhang2008}, we considered the finite
horizon $H_\infty$ control for the following affine nonlinear system
\begin{equation}
\left\{
\begin{array}{l}
x_{k+1}=f(x_k)+g(x_k)u_k+h(x_k)v_k\\
\ \ \ \ \  \ \ \hspace{0.5cm}  +[f_1(x_k)+g_1(x_k)u_k+h_1(x_k)v_k]\omega_k, \\
z_k= \left[
\begin{array}{cc}
m(x_k)\\
u_k
\end{array}
\right], \
 x_0\in {\cal R}^n.
\end{array}
\label {eq iih13_2.2.3dfd} \right.
\end{equation}
The references \cite{ShenWangShuWei2009} and \cite{lishi2012}
discussed the $H_\infty$ filtering design for some uncertain
discrete-time affine nonlinear systems with  time delays by means of
Hamilton-Jacobi inequalities  or matrix inequalities.

However, there are still some essential  difficulties in nonlinear
stochastic $H_\infty$ control design due to the following reasons:

$\bullet$  Even for affine nonlinear discrete-time multiplicative
noise systems (a special class of nonlinear stochastic systems), in
order to separate the control input
 $u$ from unknown exogenous disturbance $v$, the selection of the Lyapunov candidate function has to be  a
 quadratic function, which often leads to conservative results
 \cite{zhangxiechen2017}.

$\bullet$   Because the  Hamilton-Jacobi inequality depends on the
supremum of a conditional mathematical expectation function (see (8)
of \cite{BermanShaked2006}) or  the mathematical expectation of the
state trajectory (see (30) of \cite {luzhang2008}),  which makes
  the given  $H_\infty$ controller be not easily constructed. So
  the   general discrete-time nonlinear stochastic $H_\infty$ theory  merits further study,  and new methods should be introduced in this field.

$\bullet$  Even for the affine nonlinear system (\ref{eq
iih13_2.2.3dfd}), as said in  \cite{zhangxiechen2017},  the
completing the squares technique is no longer applicable except for
special  quadratic Lyapunov functions. Different from linear system
case, the nonlinear discrete system cannot be iterated. In addition,
different from It\^o systems where an infinitesimal generator ${\cal
L}V(x)$ can be used, how to give practical $H_\infty$  criteria for
general nonlinear discrete-time stochastic systems  which are not
dependent on the mathematical expectation of the trajectory is a
challenging problem.

This paper will make a contribution to the $H_\infty$ theory of
general nonlinear discrete-time stochastic systems.  It is
well-known that the bounded real lemma plays a key role in the study
of $H_\infty$ control, so we will first establish a bounded real
lemma for the following discrete-time nonlinear stochastic
state-disturbance system
\begin{eqnarray}\label{eq-system-1}
\left\{\begin{array}{l}
  x_{k+1}=f(x_k,\omega_k)+g(x_k,\omega_k)v_k,\\
  z_k=\left[\begin{array}{l}
  m(x_k)\\
  m_1(x_k)v_k
  \end{array}\right], x_0\in\myR^n, k=0,1,2,\cdots
  \end{array}
  \right.
\end{eqnarray}
where $f:\myR^n\times\myR^d\rightarrow \myR^n$,
$g:\myR^n\times\myR^d\rightarrow \myR^{n\times n_v}$,
$m:\myR^n\rightarrow \myR^{n_m}$ and $m_1:\myR^n\rightarrow
\myR^{(n_z-n_m)\times n_v}$ are measurable vector/matrix-valued
functions. $x_k$ , $v_k$ and $z_k$ represent respectively the system
state, external disturbance and the regulated output with
appropriate dimensions. Throughout this paper,
 $\{\omega_k\}_{k\in\myN}$ is a sequence of independent $d$-dimensional random variables with an identical distribution defined on the complete probability
 space
$(\Omega,\myF,\myP)$, and the corresponding filtration is
 $\myFb=\{\myF_k\}_{k\in\myN}$,  where $\myF_k$ is the  $\sigma$-field  generated by $\omega_0,\cdots,\omega_{k-1}$.
 Based on the
obtained bounded  real lemma, we pay our attention to the $H_\infty$
control of the following controlled system
\begin{eqnarray}\label{eq-system-control-H-infty}
  \left\{\begin{array}{l}
  x_{k+1}=f(x_k,u_k,\omega_k)+g(x_k,\omega_k)v_k\\
  z_k=\left[\begin{array}{c}
    m(x_k,u_k)\\
    m_1(x_k)v_k
  \end{array}\right], x_0\in\myR^n, k=0,1,2,\cdots
  \end{array}\right.
\end{eqnarray}
 where $f:\myR^n\times\myR^{n_u}\times\myR^d\rightarrow\myR^n$ and
 $m:\myR^n\times\myR^{n_u}\rightarrow\myR^{n_m}$  are respectively measurable vector-valued functions. $\{u_k, k=0,1,2,\cdots\}$ is the control input sequence.
$\{v_k, k=0,1,2,\cdots\}$ and $\{u_k, k=0,1,2,\cdots\}$ are  adapted
sequences with respect to $\{\myF_k, k=0,1,2,\cdots\}$.

 For affine systems with multiplicative noises, when using the method of completing the squares as used in \cite{BermanShaked2006}, the usual conditions are supposed that $V(x)$ has
 the form of quadratics $V(x)=x^TPx$ or is twice differentiable which will be used in Taylor's expansion, see \cite{ZhangChen2006} and \cite{ShenWangShuWei2009}.
 The main purposes of those assumptions are to separate $v$ from other variables(eg.$x$ or $u$). The same difficulty which is always the main one, also exists in solving $H_\infty$
 problems of stochastic nonlinear system \eqref{eq-system-1} and \eqref{eq-system-control-H-infty}. Concretely, for system \eqref{eq-system-1}, separating $v$ from $x$ is the key that
 will solve $H_\infty$ problems to obtain some important results such as well known bounded real lemmas; and for system \eqref{eq-system-control-H-infty}, separating $v$ from $u$ and $x$
  is also the key problem in designing $H_\infty$ controller. In order to overcome those difficulties of dividing-variables, we find that the following properties of convex function
  $V:\myR^n\rightarrow\myR$
$$
V(\alpha x+(1-\alpha)y)\leq \alpha V(x)+(1-\alpha)V(y), \
\alpha\in[0,1], x, y\in\myR^n
$$
can be used in the analysis of $H_\infty$ control problems to
separate $v$ from $x$ or $u$. Based on this idea, we introduce a
convex method to discuss the $H_\infty$ control problems of system
 \eqref{eq-system-1} and \eqref{eq-system-control-H-infty}.

This paper is organized as follows: In section 2, the
 stability theory for discrete-time nonlinear systems and martingale  properties are retrospected,
 which will be used in the discussion of $H_\infty$ control. In section 3, the  internal stability  and
 external stability for system \eqref{eq-system-1} are discussed. Based on the convex properties of the auxiliary Lyapunov function, the bounded real lemma
 for system \eqref{eq-system-1} is obtained.
 In section 4, the state-feedback $H_\infty$ control is discussed via the convex analysis method, and then the state-feedback
 $H_\infty$ controller is designed. In section 5, numerical simulations are given  to show the validity  of the obtained results.

Throughout this paper, we adopt the following notations:

$\myR$: the set of all real numbers; $\myR^+$: the set of all
positive real numbers including  $0$; $\myR^n$: the $n$-dimensional
real vector space with the norm
$$
|x|=\sqrt{\sum\limits_{i=1}^nx_i^2}
$$
for $x=(x_1,\cdots,x_n)^T\in\myR^n$; $\myR^{m\times n}$: the set of
all real  $m\times n$ matrices; $\myN$: the set of all positive
integers including $0$; $n_v$: the dimension of vector $v$;
$\myS^n(\myR)$: the set of all $n\times n$ symmetric matrices;
$\myS^n_+(\myR)$: the set of all real positive definite symmetric
matrices; $\bar{\sigma}(Q)(\underline{\sigma}(Q))$: the
maximum(minimum) eigenvalue of $Q\in\myS^n(\myR)$; $P\geq0$($P>0$):
the symmetric matrix $P$ is positive semi-definite (definite);
$L^2(\Omega,\myF_k;\myR^{n_v})$: the $\myF_k$-measurable
second-order moment  random variable space with the norm
$$
\|\xi\|_{L^2_{\myF_k}}=\sqrt{\myE|\xi|^2}<\infty;
$$
$\myLiv$: the space of  stochastic sequence $v=\{v_k\}_{k\in\myN}$
with the norm
$$
\|v\|_{l_\infty^2}=\sqrt{\myE\big[\sum\limits_{k=0}^\infty|v_k|^2\big]}<\infty,
$$
where $v_k\in L^2(\Omega,\myF_k;\myR^{n_v})$, $k\in\myN$.

%

\section{Preliminaries\label{sec:UD}}

Throughout this paper, let $(\Omega,\myF,\myP)$ be a complete
probability space and $\{\omega_k\}_{k\in\myN}$  is an
$\myR^d$-valued independent random variable sequence. Denote $\myNz$
 the event set that has zero probability. Let $\myF_k$  the $\sigma$-field
generated by $\omega_0,\omega_1,\cdots,\omega_{k-1}$, {\it i.e.,}
\begin{equation}\begin{array}{l}
\nonumber
\myF_k=\sigma\{\omega_0,\omega_1,\cdots,\omega_{k-1}\}\vee\myNz,k\in\myN.
\end{array}\end{equation}
and  $\myF_0=\{\emptyset,\Omega\}$($\emptyset$ is the empty set,
$\Omega$ is the sample space). Obviously, $\myF_{k-1}\subset\myF_k$,
and we set $\myFb=\{\myF_k\}_{k\in\myN}$. Now, we first review some
results on the conditional expectation which will be used latter.
The following lemma is the special case of Theorem 6.4 in
\cite{Kallenberg2002}.
\begin{lemma}\label{lem-property-conditional-expectation}
  If $\myR^d$-valued random variable $\eta$ is independent of  the $\sigma-$field $\myG\subset \myF$, and $\myR^n-$valued random
  variable $\xi$ is $\myG-$measurable, then, for every bounded function $f:\myR^n\times\myR^d\rightarrow\myR$, there exists
  $$
  \myE[f(\xi,\eta)|\myG]=\myE[f(x,\eta)]_{x=\xi}\quad \myas
  $$
\end{lemma}
We firstly retrospect the stability theory for the following
discrete-time stochastic system
\begin{equation}\begin{array}{l}\label{eq-system-pre}
  \left\{\begin{array}{l}
    x_{k+1}=F_k(x_k,\omega_k),\\
    x_0\in\myR^n,
  \end{array}
  \right.
\end{array}\end{equation}
where $F_k:\myR^n\times\myR^d\rightarrow\myR^n$ is a measurable
function with $F_k(0,\cdot)\equiv0$. From the definition of system
\eqref{eq-system-pre}, it is easy to see that the solution $x_k$ is
$\myF_k-$adapted.
 Denote $x(k;s,x)$ or $x_k^{s,x}$ the solution of \eqref{eq-system-pre}  at time $k$ with the initial  state  $x\in\myR^n$  starting at $s\in\myN$, where $k\geq s$.

\begin{definition}
 The equilibrium solution $x_k\equiv0$ of \eqref{eq-system-pre} is said to be

(1) almost surely asymptotically stable, if, for all $x_0\in\myR^n$,
$s\geq0$
  \begin{equation}
   \myP\bigg\{\lim\limits_{k\rightarrow\infty}x(k;s,x_0)=0\bigg\}=1.
  \end{equation}

(2) asymptotically $p$-stable, if
\begin{equation}
  \lim\limits_{k\rightarrow \infty}\myE[|x(k;s,x_0)|^p]=0;
\end{equation}
\end{definition}
The following lemma is the LaSalle-type  theorem for the
discrete-time stochastic system \eqref{eq-system-pre}; see
\cite{ZhangLinChen2017} for details.
\begin{lemma}\label{th-LaSalle}
  Suppose $W:\myR^n\rightarrow \myR^+$ is a positive function and $V_k:\myR^n\rightarrow \myR^+$, $k\in\myN$, are  the Lyapunov
  functions satisfying
  \begin{equation}
  \begin{array}{l}
    \myE[V_{k+1}(F_k(x,\omega_k))]-V_k(x)\leq \gamma_k-W(x),\quad\forall x\in\myR^n, k\in\myN,
  \end{array}
  \label{eq-Lyapunov-inequality}
  \end{equation}
  $$
  \sum\limits_{k=0}^\infty\gamma_k< \infty
  $$
  and
  \begin{equation}\begin{array}{l}\label{eq-V-inf-infty}
    \liminf\limits_{|x|\rightarrow\infty}\inf\limits_{k\in\myN}V_k(x)=\infty.
  \end{array}\end{equation}
  $\{x_k\}_{k\in \myN}$ is the solution sequence of
  \eqref{eq-system-pre}.
  Then
  $$\lim\limits_{k\rightarrow\infty}V_k(x_k) \quad\mbox{exists and is finite almost surely,}$$
  and
  $$\lim\limits_{k\rightarrow\infty}W(x_k)=0\quad a.s..$$
  \end{lemma}
Under the condition that $W$ is proper and continuous positive
definite, the following corollary can be obtained directly by
LaSalle-type theorem.
\begin{corollary}
  Suppose there exist a proper and continuous positive definite function   $W$ and a Lyapunov function sequence $\{V_k,k\in\myN\}$ satisfying the conditions of Lemma \ref{th-LaSalle},  then
  $$
  \lim\limits_{k\rightarrow\infty}x_k=0\quad \mbox{a.s..}
  $$
\end{corollary}

\section{A discrete-time version of the bounded real lemma\label{sec:EDEO}}

Now we consider the discrete-time system \eqref{eq-system-1}, where
$x(\cdot):=\{x_k\}_{k\in\myN}$ is the solution
 of \eqref{eq-system-1} with the initial state $x_0\in\myR^n$, $v(\cdot):=\{v_k\}_{k\in\myN}\in\myLiv$  is the
 exogenous disturbances to be rejected,
 and $z(\cdot):=\{z_k\}_{k\in\myN}$ is the regulated  output.
 Without loss of generality, we also assume that $0$ is the equilibrium  of $f$ and $m$, i.e.,
 $f(0,\cdot)\equiv0,m(0)=0$. In this section, we denote $x(k;s,x_s,v)$ or $x_k^{s,x_s,v}$ the solution of \eqref{eq-system-1}
 with the initial state $x_s\in\myR^n$ and external disturbance  $v(\cdot)$ starting at $s\in\myN$, and denote
 the controlled  output as $z(k;s,x_s,v)$ or $z_k^{s,x_s,v}$  corresponding to
 $x_k^{s,x_s,v}$ for $k\geq s$. Throughout the paper, we assume that all  random
variables such as $V(x_k)$ and $V_k(x_k)$ are
 elements in $L^1(\Omega,\myF,\myP)$, i.e., $\myE[V(x_k)]<\infty$ and $\myE[V_k(x_k)]<\infty$.

\begin{definition}\label{def-internal-stable}
  The system \eqref{eq-system-1} is called internally stable if there exists
  $c>0$ such that
  \begin{equation*}
    \sum\limits_{k=0}^\infty \myE|z_k^{0,x_0,0}|^2\leq c|x_0|^2,\ \ x_0\in \myR^n,
  \end{equation*}
where $z_k^{0,x_0,0}=m(x_k^{0,x_0,0})$.
\end{definition}
For every positive function $V:\myR^n\rightarrow\myR^+$ and
disturbance $v\in\myR^{n_v}$, we define the difference operator
$\Delta_v$ of system \eqref{eq-system-1} as
$$
\Delta_vV(x):=\myE[V(f(x,\omega_k)+g(x,\omega_k)v)]-V(x),x\in\myR^n.
$$
 Because we assume that $\{\omega_k\}_{k\in\myN}$ is independently  identically distributed,
 so
$$
\myE[V(f(x,\omega_k)
+g(x,\omega_k)v)]=\myE[V(f(x,\omega_{k+1})+g(x,\omega_{k+1})v)],
$$
i.e., the difference operator $\Delta_v$ is identical for all
$k\in\myN$. Specially, for $v(\cdot)\equiv0$, the operator
$\Delta_0$ reduces to
$$
\Delta_0V(x):=\myE[V(f(x,\omega_k))]-V(x),x\in\myR^n.
$$
\begin{lemma}\label{lem-stability-lemma}
  Suppose there exist  a positive function $V:\myR^n\rightarrow\myR^+$, and two positive constants $c_1>0$  and $c_2>0$, such that
  \begin{eqnarray}
    \label{eq-cond-Delta-V-1} &&\Delta_0V(x)\leq -c_1 |m(x)|^2,\\
  \label{eq-cond-Delta-V-01}&& V(x)\leq c_2|x|^2,
  \end{eqnarray}
  then system \eqref{eq-system-1} is internally stable. Moreover, if $|m(x)|$ is positive
   definite, then for every $x_0\in\myR^n$, we have
  \begin{equation}\label{eq-internal-almost-surely}
    \lim\limits_{k\rightarrow\infty} x_k^{0,x_0,0}=0\quad\mbox{a.s..}
  \end{equation}
\end{lemma}

\begin{proof}
Since
  $$
  \myE[V(x_{k+1})]-\myE[V(x_k)]=\myE\bigg[\myE\big[V(f(x_k,\omega_k))|\myF_k\big]-V(x_k)\bigg],
  $$
$x_k$ is $\myF_k$-measurable and $\omega_k$ is independent of
$\myF_k$, by Lemma \ref{lem-property-conditional-expectation}, we
have
  \begin{eqnarray*}
   \myE[V(x_{k+1})]-\myE[V(x_k)]&=&\myE\bigg\{\bigg[\myE[V(f(x,\omega_k))]-V(x)\bigg]_{x=x_k}\bigg\}\\
  &=&\myE\big[\Delta_0V(x_k)\big].
  \end{eqnarray*}
By condition  \eqref{eq-cond-Delta-V-1}, it shows that
  \begin{eqnarray*}
  \myE[V(x_{k+1})]-\myE[V(x_k)]\leq -c_1\myE|m(x_k)|^2.
  \end{eqnarray*}
 For every $N\in\myN$, taking the summation  on both sides of the above inequality for $k$ from $0$ to $N$, we obtain that
  \begin{eqnarray*}
    \myE[V(x_{N+1})]-\myE[V(x_0)]\leq-c_1\sum\limits_{k=0}^N\myE|m(x_k)|^2.
  \end{eqnarray*}
 Since $V(x)$ is a positive function, the above inequality  yields
   \begin{eqnarray}\label{eq-proof-lemma-internally-stable}
    \sum\limits_{k=0}^N\myE|m(x_k)|^2\leq\frac{1}{c_1}\myE[V(x_0)].
  \end{eqnarray}
  In view of  \eqref{eq-cond-Delta-V-01}, by letting $N\rightarrow\infty$ on the left-hand side  of \eqref{eq-proof-lemma-internally-stable},
  we have
    \begin{eqnarray}
    \sum\limits_{k=0}^\infty\myE|m(x_k)|^2\leq \frac{1}{c_1}V(x_0)\leq \frac{c_2}{c_1}|x_0|^2.
    \label {weihaizhang add}
    \end{eqnarray}
  Since $|z_k^{0,x_0,0}|=|m(x_k)|$, the internal stability is shown from (\ref{weihaizhang add}).

  As far as \eqref{eq-internal-almost-surely}, it can be obtained directly by Lemma \ref{th-LaSalle} and the positive definiteness  of the function $W(x)=c_1|m(x)|^2$.
\end{proof}

Now, we will show the converse of Lemma \ref{lem-stability-lemma}
which is characterized by the following lemma.
\begin{lemma}\label{lem-stability-lemma-inverse}
  Suppose system \eqref{eq-system-1} is internally stable. Then there exists a positive function $V:\myR^n\rightarrow \myR^+$ satisfying \eqref{eq-cond-Delta-V-1} and \eqref{eq-cond-Delta-V-01}.
\end{lemma}
\begin{proof}
For every $x\in\myR^n$, define
\begin{equation}\label{eq-proof-define-V-k}
  V_k(x)=\sum\limits_{i=k}^\infty\myE|m(x^{k,x,0}_i)|^2.
\end{equation}
Because, for every $k\in\myN$, the following fact holds:
$$
x_{k+1}^{k,x,0}=f(x,\omega_k),
$$
which implies that
$$
x_i^{k,x,0}=x_i^{k+1,x_{k+1}^{k,x,0},0}=x_i^{k+1,f(x,\omega_k),0},\quad
i> k.
$$
Using the above property for the solution of system
\eqref{eq-system-1}, we have
\begin{eqnarray}
  \nonumber\myE[V_{k+1}(f(x,\omega_k))]-V_k(x)&=&\myE\left\{\left[\sum\limits_{i=k+1}^\infty\myE|m(x^{k+1,y,0}_i)|^2\right]_{y=f(x,\omega_k)}\right\}\\
 && \nonumber-\sum\limits_{i=k}^\infty\myE|m(x^{k,x,0}_i)|^2\\
  \nonumber&=&\sum\limits_{i=k+1}^\infty\myE|m(x^{k,x,0}_i)|^2-\sum\limits_{i=k}^\infty\myE|m(x^{k,x,0}_i)|^2\\
 \nonumber &=&-\myE|m(x^{k,x,0}_k)|^2=-|m(x)|^2.
\end{eqnarray}
Hence, we obtain  the following equations  for all $k\in\myN$:
\begin{eqnarray}\label{eq-V_k-Delta}
  \myE[V_{k+1}(f(x,\omega_k))]-V_k(x)=-|m(x)|^2,  \ k\in\myN.
\end{eqnarray}
Below, we prove that for any  $k\in\myN$, the following holds:
$$
V_k(x)=V_{k+1}(x), \ \ x\in\myR^n.
$$
Because, for every $k\in\myN$, $x_{k+1}^{k,x,0}=f(x,\omega_k)$,
$x_{k+2}^{k+1,x,0}=f(x,\omega_{k+1})$, and $\omega_k$ and
$\omega_{k+1}$ are  independently identically distributed, which
implies that $x_{k+1}^{k,x,0}$ and $x_{k+2}^{k+1,x,0}$ are also
identically distributed. So
$$
\myE |m(x_{k+1}^{k,x,0})|^2=\myE |m(x_{k+2}^{k+1,x,0})|^2.
$$
Similarly, the following relationship holds:
$$
\myE |m(x_{i}^{k,x,0})|^2=\myE |m(x_{i+1}^{k+1,x,0})|^2, \
i=k,k+1,\cdots
$$
By the definition of $V(x)$ in \eqref{eq-proof-define-V-k}, we have
$$
V_k(x)=V_{k+1}(x), \ \forall k\in\myN,
$$
which implies that $V_k(x)$ is identical for all $k\in\myN$.
Therefore, if  we let
\begin{equation}\label{eq-define-V}
V(x)=\sum\limits_{i=0}^\infty\myE|m(x^{0,x,0}_i)|^2,
\end{equation}
then,  by the above discussion, it follows that $V(x)=V_k(x),\forall
k\in\myN$. As so, the  equation \eqref{eq-V_k-Delta} reduces to
\begin{equation}\label{eq-proof-lemma-stability-inverse-m(x)}
\Delta_0 V(x)=-|m(x)|^2.
\end{equation}
Taking $c_1=1$, we have proved that $V$ defined by
\eqref{eq-define-V} satisfies \eqref{eq-cond-Delta-V-1}.

As far as $V(x)$ satisfies \eqref{eq-cond-Delta-V-01}, it can be
obtained directly by the internal stability of system
\eqref{eq-system-1} and Definition \ref{def-internal-stable}.
\end{proof}

By the equations \eqref{eq-define-V} and
\eqref{eq-proof-lemma-stability-inverse-m(x)}, we  have the
following corollary.
\begin{corollary}
Suppose system \eqref{eq-system-1} is internally stable. Then there
exists a positive function $V:\myR^n\rightarrow \myR^+$ satisfying
\eqref{eq-proof-lemma-stability-inverse-m(x)}. Moreover, there also
exists
\begin{eqnarray}\label{eq-V-geq}
V(x)\geq |m(x)|^2.
\end{eqnarray}
\end{corollary}
\begin{proof}
Obviously, it only remains to show  that \eqref{eq-V-geq}. By
definition of $V(x)$  in  \eqref{eq-define-V}, we have
$$
V(x)\geq |m(x_0^{0,x,0})|^2.
$$
In view of  the fact that $m(x_0^{0,x,0})=m(x)$, \eqref{eq-V-geq} is
hence  proved.
\end{proof}
Combining Lemma \ref{lem-stability-lemma} and Lemma
\ref{lem-stability-lemma-inverse}, the following proposition
\ref{prop-stability} is obtained, which presents  a necessary  and
sufficient  condition  of the  internal stability of system
\eqref{eq-system-1}. Denote
  \begin{eqnarray}
    \myH_0(V(x)):=\myE[V(f(x,\omega_0))]-V(x)+|m(x)|^2.
  \end{eqnarray}
\begin{proposition}\label{prop-stability}
  System \eqref{eq-system-1} is internally stable if and only if there exist a positive function $V:\myR^n\rightarrow\myR^+$ and a positive constant $c_2>0$
 such that
  \begin{eqnarray}
  \label{eq-cond-H-0-1}&&V(x)\leq c_2|x|^2, \ \ \forall x\in\myR^n, \\
  \label{eq-cond-H-0-2} && \myH_0(V(x))\leq0, \ \  \forall
  x\in\myR^n.
  \end{eqnarray}
\end{proposition}
\begin{definition}
  The system \eqref{eq-system-1} is said to be externally stable or $l^2$-input-output stable if, for every $v(\cdot)\in \myLiv$,
  $$
 z(\cdot)=\{z_k^{0,0,v}\}_{k\in\myN}\in\myLiz,
  $$
  and there exists a positive real number $\gamma>0$ such that
  \begin{eqnarray}\label{eq-definition-external-z-norm}
    \|z(\cdot)\|_{l^2_\infty}\leq\gamma\|v(\cdot)\|_{l^2_\infty},
   \nonumber \forall v(\cdot)\in\myLiv,
  \end{eqnarray}
  or  equivalently,
    \begin{eqnarray}\label{eq-definition-external-sum-z}
    \sum\limits_{k=0}^\infty\myE[|z_k^{0,0,v}|^2]\leq\gamma^2\sum\limits_{k=0}^\infty\myE[|v_k|^2].
  \end{eqnarray}
\end{definition}
\begin{remark}
  Suppose $\gamma$ is a given positive real number. If inequality \eqref{eq-definition-external-z-norm} or \eqref{eq-definition-external-sum-z} holds,
  system \eqref{eq-system-1} is also said to have $l_2$-gain less than or equal to $\gamma$ \cite{LinByrnes1996}.  Moreover, suppose that system \eqref{eq-system-1}
  is externally stable. Define an  operator
\begin{eqnarray*}
  \myL:\myLiv\rightarrow\myLiz
\end{eqnarray*}
by
\begin{eqnarray*}
  \myL(v)=z(\cdot,0,0,v),v\in\myLiv,
\end{eqnarray*}
then operator $\myL$ is called the perturbation operator of
\eqref{eq-system-1}. Its norm is defined as
\begin{eqnarray}
  \|\myL\|=\sup\limits_{0\neq v(\cdot)\in\myLiv}\frac{\|z(\cdot,0,0,v)\|_{\myLiz}}{\|v\|_{\myLiv}}.
\end{eqnarray}
So, on  one hand, $\|\myL\|$ is a measure of the $l_2$-gain of
system \eqref{eq-system-1}, but on the other hand, it  is also a
measure of the worst case  effect that the stochastic disturbance
$v$ may have on the controlled output $z$. Therefore, it is
important to find a way to determine or estimate the norm
$\|\myL\|$.
\end{remark}

\begin{proposition}\label{prop-stable-external}
  Suppose, for $\gamma>0$, there exist   a convex positive function $V:\myR^n\rightarrow \myR^+$ and a  real number $\beta>1$, such that
 \begin{eqnarray}
    \label{eq-H-1-ineq-1} \myH_1(V(x),\beta):=\frac{1}{\beta}\myE[V(\beta f(x,\omega_0))]-V(x)+|m(x)|^2\leq 0, \\
   G_\beta(V(x))\leq\gamma^2, \forall x\in\myR^n, V(0)=0,\label{eq-H-1-ineq-2}
 \end{eqnarray}
 where $G_\beta(V(x))$ is defined by
 \begin{eqnarray}
    \label{eq-G-beta}
    G_\beta(V(x)):=\sup\limits_{0\neq v\in\myR^{n_u}}\left\{\frac{(\beta-1)}{\beta}
    \frac{\myE\big[V(\frac{\beta}{\beta-1}g(x,\omega_0)v)\big]}{|v|^2}+\frac{|m_1(x)v|^2}{|v|^2}\right\}.
 \end{eqnarray}
 Then $\|\myL\|\leq\gamma$. Moreover, if $V$ satisfies \eqref{eq-cond-Delta-V-01}, then system \eqref{eq-system-1} is also internally stable.
\end{proposition}
\begin{proof}
Let $\alpha=1/\beta$, then $0<\alpha<1$. By the convexity of $V$, it
follows
  \begin{eqnarray*}
  \Delta_vV(x)&=&\myE[V(f(x,\omega_k)+g(x,\omega_k)v)]-V(x)\\
  &&= \myE\bigg[V(\alpha\frac{1}{\alpha}f(x,\omega_k)+(1-\alpha)\frac{1}{1-\alpha}g(x,\omega_k)v)\bigg]-V(x)\\
  &&\leq \alpha \myE\bigg[V(\frac{1}{\alpha}f(x,\omega_k))\bigg]+(1-\alpha)\myE\bigg[V(\frac{1}{1-\alpha}g(x,\omega_k)v)\bigg]\\
  &&\quad-V(x)+|m(x)|^2-\gamma^2|v|^2-|m(x)|^2+\gamma^2|v|^2.\\
  &&\leq\myH_1(V(x),\beta)+[G_\beta(V(x))-\gamma^2]|v|^2-|z|^2+\gamma^2|v|^2.
  \end{eqnarray*}
By conditions of \eqref{eq-H-1-ineq-1} and \eqref{eq-H-1-ineq-2}, it
follows that
  \begin{eqnarray*}
  &&\Delta_vV(x)\leq-|z|^2+\gamma^2|v|^2.
  \end{eqnarray*}
Denote $x_k$ the solution of \eqref{eq-system-1} with initial state
$x_0=0$ for $v(\cdot)\in\myLiv$, $z_k$ is the corresponding output.
Then, we have
\begin{eqnarray*}
  \left\{\myE[V(f(x,\omega_k)+g(x,\omega_k)v)]-V(x)\right\}_{x=x_k,v=v_k}\leq-|z_k|^2+\gamma^2|v_k|^2,
  \end{eqnarray*}
Since $x_k$ and $v_k$ are  $\myF_k$-measurable, by Lemma
\ref{lem-property-conditional-expectation}, the above inequality can
also be written as
\begin{eqnarray*}
  \myE[V(f(x_k,\omega_k)+g(x_k,\omega_k)v_k)|\myF_k]-V(x_k)
  \leq-|z_k|^2+\gamma^2|v_k|^2,
\end{eqnarray*}
i.e.
\begin{eqnarray*}
  && \myE[V(x_{k+1})|\myF_k]-V(x_k)\leq-|z_k|^2+\gamma^2|v_k|^2,
\end{eqnarray*}
Taking the mathematical expectation on both sides of the above
inequality, we have
\begin{eqnarray}\label{eq-proof-prop-1}
  \myE[V(x_{k+1})]-\myE[V(x_{k})]\leq-\myE[|z_k|^2]+\gamma^2\myE[|v_k|^2].
  \end{eqnarray}
For every $N\in\myN$, taking a summation  on both sides of
\eqref{eq-proof-prop-1} from $k=0$ to $k=N$, we have
  \begin{eqnarray*}
  &\myE[V(x_{N+1})]-\myE[V(x_0)]\leq-\sum\limits_{k=0}^N\myE[|z_k|^2]+\gamma^2\sum\limits_{k=0}^N\myE[|v_k|^2].
  \end{eqnarray*}
Since $V(x_0)=V(0)=0$, and $V(x)\geq0$, we obtain that
  \begin{eqnarray*}
     \sum\limits_{k=0}^N\myE[|z_k|^2]\leq \gamma^2\sum\limits_{k=0}^N\myE[|v_k|^2].
  \end{eqnarray*}
Let $N\rightarrow \infty$, we get
  \begin{eqnarray*}
     \sum\limits_{k=0}^\infty\myE[|z_k|^2]\leq \gamma^2\sum\limits_{k=0}^\infty\myE[|v_k|^2].
  \end{eqnarray*}
  This proves that \eqref{eq-system-1} is externally stable and $\|\myL\|\leq\gamma$.

  Now, we will prove that system \eqref{eq-system-1} is also internally stable. Since
  \begin{eqnarray*}
   V(x)&=&V(\alpha\beta x+(1-\alpha)0)\leq \alpha V(\beta x)+(1-\alpha)V(0)\\
   &=&\alpha V(\beta x),
  \end{eqnarray*}
 this implies
  $$
  V(\beta x)\geq \beta V(x).
  $$
  By \eqref{eq-H-1-ineq-1} and above inequality, we obtain
  $$
   \myE[V(f(x,\omega_0))]- V(x)+|m(x)|^2\leq 0,
  $$
  i.e.
  $$
  \myH_0(V(x))\leq0.
  $$
  By Proposition \ref{prop-stability}, we proves that system \eqref{eq-system-1} is internally stable.
\end{proof}
\begin{remark}
  Denote $\myfunV$ a set of all positive convex functions defined on $\myR^n$ satisfying \eqref{eq-cond-Delta-V-01} and
\begin{eqnarray*}
    \myBeta=\left\{\begin{array}{r}(\beta,V):\beta>1\mbox{ and }V\in\myfunV  \mbox{ satisfy \eqref{eq-H-1-ineq-1}}
    \end{array}\right\}.
\end{eqnarray*}

  Define
  \begin{eqnarray}\label{eq-def-gamm-star}
      \gamma^{*^2}=\inf\limits_{(\beta,V)\in\myBeta}\sup\limits_{x\in\myR^n}G_\beta(V(x)).
  \end{eqnarray}
  From the proof of Proposition \ref{prop-stable-external}, we can see that $\|\myL\|\leq\gamma^*$. This can be used to estimate the upper bound of operator norm $\|\myL\|$, though $\gamma^*$ given by \eqref{eq-def-gamm-star} is not necessarily the best one. But it is the locally best one, this is because that $V$ is confined to $\myfunV$  which is a subset of convex functions.
\end{remark}

%


In order to induce the bounded real lemma for system
\eqref{eq-system-1}, we introduce the definition of convexity of
vector-valued function as following.
\begin{definition}
  Let $f_0:\myR^n\rightarrow\myR^{n_0}$ and $h:\myR^{n_0}\rightarrow\myR$. The vector-valued function $f_0$ is said convex with respect to $h$, or is called $h-$convex if the compound function $h\circ f_0:\myR^n\rightarrow \myR$ is convex, i.e., for every $0<\alpha<1$ and $x,y\in\myR^n$, there exists
  \begin{equation}
    h(f_0(\alpha x+(1-\alpha)y))\leq \alpha h(f_0(x))+(1-\alpha) h(f_0(y)).
  \end{equation}
\end{definition}

\begin{remark}
 The definition of $h-$convexity can be seen as an extension of logarithmic convexity used in \cite{Montel1928}.
\end{remark}

In this paper, the following assumption is needed and will be used
in the latter discussion.

{\bf ($A_1$):} For every $w\in\myR^d$, $m(\cdot):\myR^n\rightarrow
\myR^{n_m}$ and  $m\circ f(\cdot,w):\myR^n\rightarrow\myR^{n_m}$ are
$h-$convex, where $h:\myR^{n_m}\rightarrow\myR$ is defined by
$h(y)=|y|^2,y\in\myR^{n_m}$.

\begin{lemma}\label{lem-Convex-V}
Suppose Assumption ($A_1$) holds and system \eqref{eq-system-1} is
internally stable. Then $V:\myR^n\rightarrow \myR^+$ defined by
\eqref{eq-define-V} is a convex function.
\end{lemma}
\begin{proof}
 Let $x_i^{0,x,0}$ the solution of system \eqref{eq-system-1} starting at $k=0$ with initial state $x\in\myR^n$ for $v(\cdot)=0$. Since, for every $0<\alpha<1$ and $x,y\in\myR^n$
  \begin{eqnarray*}
    x^{0,\alpha x+(1-\alpha)y,0}_0=\alpha x+(1-\alpha)y,
  \end{eqnarray*}
  applying the $h-$convexity of $m(\cdot)$ and $m\circ f$, we have
 \begin{eqnarray*}
   |m(x^{0,\alpha x+(1-\alpha)y,0}_0)|^2&=&|m(\alpha x+(1-\alpha)y)|^2
 \leq \alpha |m(x)|^2+(1-\alpha)|m(y)|^2\\
   &=&\alpha |m(x^{0,x,0}_0)|^2+(1-\alpha)|m(x^{0,y,0}_0)|^2,
 \end{eqnarray*}
 and
 \begin{eqnarray*}
   |m(f(x^{0,\alpha x+(1-\alpha)y,0}_0))|^2\leq \alpha |m(f(x^{0,x,0}_0))|^2+(1-\alpha)|m(f(x^{0,y,0}_0))|^2.
 \end{eqnarray*}

 Now we use the inductive method to prove that, for all $k\in\myN$, the following two inequalities are true:
  \begin{eqnarray}\label{eq-proof-convex-1}
   \nonumber\myE[|m(x^{0,\alpha x+(1-\alpha)y,0}_k)|^2]\leq\alpha \myE[|m(x^{0,x,0}_k)|^2]+(1-\alpha)\myE[|m(x^{0,y,0}_k)|^2]
 \end{eqnarray}
 and
 \begin{eqnarray}\label{eq-proof-convex-2}
   \myE[|m(f(x^{0,\alpha x+(1-\alpha)y,0}_k))|^2]\leq \alpha \myE[|m(f(x^{0,x,0}_k))|^2]
   +(1-\alpha)\myE[|m(f(x^{0,y,0}_k))|^2].
 \end{eqnarray}

 Firstly, for $k=0$, by the just above discussions, we see that \eqref{eq-proof-convex-1} and \eqref{eq-proof-convex-2} are true.

 Suppose, for $k\leq i$, the inequalities of  \eqref{eq-proof-convex-1} and \eqref{eq-proof-convex-2} are true. Then,
 for $k=i+1$, keeping $m\circ f$ is $h-$convex in mind, we have
 \begin{eqnarray*}
   \myE[|m(x^{0,\alpha x+(1-\alpha)y,0}_{i+1})|^2]&=&\myE[|m(f(x^{0,\alpha x+(1-\alpha)y,0}_{i},\omega_i))|^2]\\
   &\leq& \alpha\myE[|m(f(x^{0,x,0}_i,\omega_i))|^2]+(1-\alpha)\myE[|m(f(x^{0,y,0}_i,\omega_i))|^2]\\
   &\leq &\alpha\myE[|m(x^{0,x,0}_{i+1})|^2]+(1-\alpha)\myE[|m(x^{0,y,0}_{i+1})|^2].
 \end{eqnarray*}
Similarly, we can prove that \eqref{eq-proof-convex-2} is true for
$k=i+1$. By induction, we prove that \eqref{eq-proof-convex-1} and
\eqref{eq-proof-convex-2} are true.

 For every $N\in\myN$, taking summation on both sides of \eqref{eq-proof-convex-1} for $k$ from $0$ to $N$, we obtain
  \begin{eqnarray*}
   \sum\limits_{k=0}^N \myE[|m(x^{0,\alpha x+(1-\alpha)y,0}_k)|^2]\leq\alpha \sum\limits_{k=0}^N\myE[|m(x^{0,x,0}_k)|^2]+(1-\alpha)\sum\limits_{k=0}^N\myE[|m(x^{0,y,0}_k)|^2].
 \end{eqnarray*}
 Since system \eqref{eq-system-1} is internally stable, together with definition of $V(x)$ by \eqref{eq-define-V}, when let $N\rightarrow\infty$, we get
 \begin{eqnarray*}
    V(\alpha x+(1-\alpha)y)\leq \alpha V(x)+(1-\alpha)V(y),
 \end{eqnarray*}
which shows that $V(x)$ is convex. This ends the proof.
\end{proof}

We now will show that under some proper conditions, an internally
stable system \eqref{eq-system-1} is also externally stable. In the
rest of this section, the following assumptions are needed.

 {\bf ($A_2$):} $m_1(x)$ and $\myE[g(x,\omega_k)^Tg(x,\omega_k)]$ are bounded.

 {\bf ($A_3$):} For internally stable system \eqref{eq-system-1}, there exist two continuous positive functions $C_1,C_2:\myR^+\rightarrow\myR^+$ with $C_1(1)<1$ such that
\begin{eqnarray}\label{eq-assumption-V}
 C_2(\beta) \bar{V}(x)\leq \myE[\bar{V}(\beta f(x,\omega_k))]\leq C_1(\beta) \bar{V}(x),
\end{eqnarray}
where $\bar{V}:\myR^n\rightarrow\myR^+$ is defined by Lemma
\ref{lem-stability-lemma}.

\begin{lemma}\label{lem-internally-stable-externally-stable}
 Under Assumptions  ($A_1$), ($A_2$) and ($A_3$), suppose system \eqref{eq-system-1} is internally stable, then \eqref{eq-system-1} is
 externally stable. Moreover, there exist $\gamma>0$ and a positive function $V:\myR^n\rightarrow \myR^+$ such that \eqref{eq-H-1-ineq-1} and \eqref{eq-H-1-ineq-2} hold.
\end{lemma}
\begin{proof}
Since system \eqref{eq-system-1} is internally stable, by Lemma
\ref{prop-stability}, take $\bar{V}$ defined by \eqref{eq-define-V},
then
  \begin{eqnarray*}
    \Delta_0\bar{V}(x)\leq -|m(x)|^2,
  \end{eqnarray*}
  i.e.,
  \begin{eqnarray}\label{eq-proof-necessary-V-tilde}
    \myE[\bar{V}(f(x,\omega_k))]-\bar{V}(x)+|m(x)|^2\leq 0.
  \end{eqnarray}
By Assumption ($A_3$), $C_1(\cdot)$ is continuous and $C_1(1)<1$, so
 $$
 \lim\limits_{\beta\rightarrow+1}(\beta-C_1(\beta))=1-C_1(1)>0.
 $$
 This implies that there exits $\beta_0>1$ such that
 \begin{eqnarray}\label{eq-define-beta-0}
   \beta_0-C_1(\beta_0)>0.
 \end{eqnarray}
 Taking
 \begin{eqnarray}
 \label{eq-define-p-0}
 q_0=\frac{C_1(\beta_0)(1-C_2(1))}{\beta_0-C_1(\beta_0)}, \ \
  p_0=\frac{q_0\beta_0}{C_1(\beta_0)},
 \end{eqnarray}
 it is easy to check $p_0>q_0>0$. Let
 $$
 V(x)=p_0\bar{V}(x).
 $$
Applying \eqref{eq-assumption-V} in Assumption ($A_3$), we have
  \begin{eqnarray*}
   \myE[\bar{V}(f(x,\omega_k))]-\bar{V}(x)&\geq& (C_2(1)-1)\bar{V}(x)=q_0\bar{V}(x)-p_0\bar{V}(x)\\
   &\geq& \frac{q_0}{C_1(\beta_0)}\myE[\bar{V}(\beta_0 f(x,\omega_k))]-p_0\bar{V}(x)\\
   &=&\frac{1}{\beta_0} \frac{q_0\beta_0}{C_1(\beta_0)}\myE[\bar{V}(\beta_0 f(x,\omega_k))]-p_0\bar{V}(x)\\
   &=&\frac{1}{\beta_0}\myE[ p_0\bar{V}(\beta_0 f(x,\omega_k))]-p_0\bar{V}(x)\\
   &=&\frac{1}{\beta_0}\myE[V(\beta_0 f(x,\omega_k))]-V(x).
  \end{eqnarray*}
Keeping inequality \eqref{eq-proof-necessary-V-tilde} in mind, we
obtain
$$
\frac{1}{\beta_0}\myE[V(\beta_0 f(x,\omega_k))]-V(x)+|m(x)|^2\leq0.
$$
This proves that $V(x)=p_0\bar{V}(x)$ satisfies
\eqref{eq-H-1-ineq-1}.

Now, we prove that $V(x)$ also satisfies \eqref{eq-H-1-ineq-2}. By
Assumption ($A_2$), we have
\begin{eqnarray*}
 && \sigma_{\max,g}=
 \sup\limits_{x\in\myR^n}\bar{\sigma}\big(\myE\big[g(x,\omega_k)^Tg(x,\omega_k)\big]\big)
 <\infty,\\
 && \sigma_{\max,m_1}=\sup\limits_{x\in\myR^n}\bar{\sigma}\big(m_1(x)^Tm_1(x)\big)<\infty.
\end{eqnarray*}
Since $\bar{V}$ satisfies inequality  \eqref{eq-cond-Delta-V-01}, we
have
\begin{eqnarray*}
 G_{\beta_0}(V(x))\leq
 \sup\limits_{0\neq v\in\myR^{n_v}}\left\{\frac{c_2\beta_0}{\beta_0-1}\frac{\myE[|g(x,\omega_k)v|^2]}{|v|^2}+\frac{|m_1(x)v|^2}{|v|^2}\right\}\leq\gamma_0,
\end{eqnarray*}
where
\begin{eqnarray}\label{eq-define-gamma-0}
\gamma_0=\frac{c_2\beta_0^2\sigma_{\max,g}}{\beta_0-1}+\sigma_{\max,m_1}.
\end{eqnarray}
Taking $\gamma\geq \gamma_0$, we show that $V$ defined by
\eqref{eq-proof-necessary-V-tilde} also satisfies
\eqref{eq-H-1-ineq-2} for $\gamma\geq\gamma_0$. By
Proposition~\ref{prop-stable-external}, we prove that system
\eqref{eq-system-1} is externally stable and there exist $\gamma>0$
and $V(x)$ satisfying \eqref{eq-H-1-ineq-1} and
\eqref{eq-H-1-ineq-2}.
\end{proof}
\begin{remark}
  The assumption ($A_3$) plays an important role in the proof of Lemma \ref{lem-internally-stable-externally-stable}. Now we give an example to show that
  conditions of inequality \eqref{eq-assumption-V} given in $A_3$ is viable. Considering the linear case of system \eqref{eq-system-1}
  with $f(x,\omega_k)=A_kx$, $g(x,\omega_k)\equiv B$, $m(x)=Mx$, $m_1(x)\equiv M_1$, where $\{A_k\}_{k\in\myN}$ is a
  independent identically distributed random matrix sequence with  $\bar{\sigma}\big(\myE[A_k^TA_k]\big)<1$. Suppose $V(x)$ has the form of $V(x)=x^TPx$,
  $P\in\myS_+^n(\myR)$, it's easy to check that $V$ satisfies \eqref{eq-assumption-V} with
  $C_1(\beta)=\bar{\sigma}(Q)\beta^2$, $C_2(\beta)=\underline{\sigma}(Q)\beta^2$ with $C_1(1)<1$.
\end{remark}
In order to show the converse of Proposition
\ref{prop-stable-external}, we first prove the following lemma.
\begin{lemma}\label{lem-internally-externally-stable-converse}
 Suppose ($A_1$), ($A_2$) and ($A_3$) hold. If system \eqref{eq-system-1} is internally stable and $\|\myL\|\leq \gamma$, then there exists a  positive
 convex function $V(x)$ satisfying \eqref{eq-cond-Delta-V-1} and
  \begin{equation}\label{eq-cond-G-0}
    G^0(V)\leq \gamma^2,
  \end{equation}
  where
  $$
  G^0(V)=\sup\limits_{0\neq v\in\myR^{n_v}}\frac{\myE[V(g(0,\omega_k)v)]+|m_1(0)v|^2}{|v|^2}.
  $$
\end{lemma}
\begin{proof}
  Since system \eqref{eq-system-1} is internally stable, by Lemma \ref{lem-stability-lemma-inverse}, there exists $V(x)$ satisfying  \eqref{eq-cond-Delta-V-1}.
  In order to prove \eqref{eq-cond-G-0}, for every given nonzero $u\in\myR^{n_v}$, we define the following process
  $$
  v_k=\left\{\begin{array}{l}
  u, \quad\mbox{\rm if} \quad k=i,\\
  0,\quad \mbox{\rm if} \quad k\neq i.
  \end{array}\right.
  $$
  $\{x_k\}$ is the solution of \eqref{eq-system-1} corresponding to $\{v_k\}$ defined by above. Then
  \begin{eqnarray*}
    \myE[V(x_{k+1})|\myF_k]-V(x_k)&=&\myH_0(V(x_k))+\myE[V(f(x_k,\omega_k)+g(x_k,\omega_k)v_k)\\
    &&\quad-V(f(x_k,\omega_k))|\myF_k]+|m_1(x_k)|^2|v_k|^2-\gamma^2|v_k|^2\\
    &&\quad-|z_k|^2+\gamma^2|v_k|^2
  \end{eqnarray*}
  Since $x_k=0,k=0,\cdots,i$ and $v_k=0,k=0,\cdots,i-1$, taking the mathematical expectation
   and  summation from $k=0$ to $k=i$ in turn, it yields that
 \begin{eqnarray*}
    \myE V(x_{i+1})&=&\sum\limits_{k=0}^i \myE[\myH_0(V(x_k))]+\myE[V(f(0,\omega_i)+g(0,\omega_i)u)\\
    &&-V(f(0,\omega_i))]+|m_1(0)u|^2-\gamma^2|u|^2- \myE |z_i|^2+\gamma^2|u|^2.
  \end{eqnarray*}
 By (\ref{eq-cond-H-0-2}) of Proposition~\ref{prop-stability},  we must have
  \begin{eqnarray*}
    \myE[V(g(0,\omega_i)u)]+|m_1(0)u|^2-\gamma^2|u|^2&=&|z_i|^2-\gamma^2|u|^2+ \myE V[g(0,\omega_i)u]\\
    &\leq& \myE
    V[g(0,\omega_i)u],
  \end{eqnarray*}
  i.e.,
  \begin{eqnarray*}
   \frac{\myE[V(g(0,\omega_i)u)]+|m_1(0)u|^2}{|u|^2}\leq\gamma^2
  \end{eqnarray*}
  for all $0\neq u\in\myR^{n_v}$, this proves \eqref{eq-cond-G-0}.
\end{proof}

Generally speaking, it is  not easy to  prove the inverse of Lemma
\ref{lem-internally-stable-externally-stable} and to obtain the
bounded real lemma for the general stochastic nonlinear system
\eqref{eq-system-1}. In order to derive the inverse of Lemma
\ref{lem-internally-stable-externally-stable} and to obtain the
bounded real lemma for system \eqref{eq-system-1}, the following
assumption is needed:

{\bf ($A_2'$):} $g(x,\omega_k)\equiv B\in \myR^{n\times n_v}$,
$m_1(x)\equiv M_1\in\myR^{(n_z-n_m)\times n_v}$. For internally
stable system \eqref{eq-system-1} and $\gamma>0$, the following
holds:
\begin{eqnarray}\label{eq-cond-G-0-beta}
 G^0(\frac{\beta_0p_0}{\beta_0-1}\bar{V})\leq\gamma^2,
\end{eqnarray}
where $\bar{V}$ is defined by \eqref{eq-define-V}, $\beta_0>1$
satisfies  \eqref{eq-define-beta-0} and $p_0>0$ is defined by
\eqref{eq-define-p-0}.

\begin{lemma}\label{lem-external-inverse}
 Suppose ($A_1$), ($A'_2$) and ($A_3$) hold. If system \eqref{eq-system-1} is internally stable and $\|\myL\|\leq \gamma$, then there exists a positive convex
 function $V(x)$ satisfying \eqref{eq-H-1-ineq-1} and \eqref{eq-H-1-ineq-2}.
\end{lemma}
\begin{proof}
By Lemma \ref{lem-internally-externally-stable-converse}, there
exists a convex function $\bar{V}(x)$ satisfying
\eqref{eq-cond-Delta-V-1} and \eqref{eq-cond-G-0}. Furthermore,
there exists $\beta_0>1$ such that $p_0>0$, where $p_0$ is defined
by \eqref{eq-define-p-0}. Let $V(x)=p_0\bar{V}(x)$. Similar to the
proof of Lemma \ref{lem-internally-stable-externally-stable},  it is
easy to prove that $V(x)$ satisfies the inequality
\eqref{eq-H-1-ineq-1}. Because
$$
G_\beta(V(x))=G^0(\frac{\beta p_0}{\beta-1}\bar{V}),
$$
which, together with Assumption ($A_2'$),  shows the inequality
\eqref{eq-H-1-ineq-2}. The proof is completed.
\end{proof}
\begin{remark}
  Comparing  \eqref{eq-H-1-ineq-2} with  \eqref{eq-cond-G-0}, we can find that \eqref{eq-H-1-ineq-2} holds  for all $x\in\myR^n$, while
   \eqref{eq-cond-G-0} holds only  at  $x=0$.
   This shows that \eqref{eq-H-1-ineq-2} implies \eqref{eq-cond-G-0-beta}, but the inverse is not always  true for general $g$ and $m_1$.
   In Assumption  ($A_2'$), $g$ and $m_1$ does not depend on $x$, which ensure that \eqref{eq-cond-G-0-beta} implies \eqref{eq-H-1-ineq-2}.
\end{remark}
Combining Lemma \ref{lem-internally-stable-externally-stable} and
Lemma \ref{lem-external-inverse}, we are in a position to obtain a
stochastic version of the bounded real lemma as follows:
\begin{theorem}\label{th-bounded-real-lemma}
  (Stochastic bounded real lemma) Under Assumptions ($A_1$), ($A_2'$) and ($A_3$), for any positive real number $\gamma 0$, the following statements are equivalent:

  {\rm(i)} The system \eqref{eq-system-1} is internally stable and $\|\myL\|\leq \gamma$.

  {\rm(ii)} There exists a convex positive function $V:\myR^n\rightarrow\myR^+$ such that \eqref{eq-H-1-ineq-1} and \eqref{eq-H-1-ineq-2} hold.
\end{theorem}
Specially, for linear case with following form

\begin{equation}\label{eq-system-linear-1}
\left\{\begin{array}{l}
  x_{k+1}=Ax_k+A_0x_k\omega_k+Bv_k\\
  z_k=\left[\begin{array}{c}
  Cx_k\\
  Dv_k
  \end{array}\right]
\end{array}\right.
\end{equation}
where $A,A_0\in\myR^{n\times n}$, $B\in\myR^{n\times n_v}$,
$C\in\myR^{m\times n}$, $D\in\myR^{(n_z-m)\times n_v}$, and
$\{\omega_k\}_{k\in\myN}$ is an  independent identical distributed
1-dimensional random variable series with $\myE[\omega_k]=0$ and
$\myE[\omega_k^2]=1$, $k\in\myN$. The following assumptions are
needed.

Similar to Proposition \ref{prop-stability}, the following lemma can
be obtained directly.
\begin{lemma}\label{lem-stability-lemma-linear}
System \eqref{eq-system-linear-1} is internally stable if and only
if there exists $P\in\myS_+^n(\myR)$ such that
  \begin{eqnarray}\label{eq-cond-P}
    A^TPA+A_0^TPA_0-P+C^TC\leq0.
  \end{eqnarray}
\end{lemma}
In order to obtain the bounded real lemma for linear system
\eqref{eq-system-linear-1}, the following assumption is needed,
which corresponds to ($A'_2$) and ($A_3$).

{\bf Assumption ($A''_2$)}: $\bar{\sigma}(A^TA+A_0^TA_0)<1$, and
$\frac{\beta_0^2p_0}{\beta_0-1}B^TPB$ $+D^TD\leq\gamma^2I_{n_v}$,
where
$$
1<\beta_0<\frac{1}{\bar{\sigma}(A^TA+A_0^TA_0)},
$$
$$p_0=\frac{\beta_0[1-\underline{\sigma}(A^TA+A_0^TA_0)]}{\beta_0-\bar{\sigma}(A^TA+A_0^TA_0)},
$$
and $P$ satisfies \eqref{eq-cond-P}.

\begin{remark}
 If system \eqref{eq-system-linear-1} is internally stable, taking $V(x)=x^TPx$, then Assumption ($A_2''$) implies Assumptions  ($A'_2$) and ($A_3$).
 It  is  easy to check that system \eqref{eq-system-linear-1} satisfies the assumptions of Theorem \ref{th-bounded-real-lemma}.
\end{remark}
Corresponding to  Theorem \ref{th-bounded-real-lemma},  the bounded
real lemma for linear system \eqref{eq-system-linear-1} is expressed
via algebraic inequalities.
\begin{theorem}\label{th-BRL-linear-our}
 Under Assumption ($A''_2$), for $\gamma>0$, the following statements are equivalent:

 {\rm(i)} The system \eqref{eq-system-linear-1} is internally stable and $\|\myL\|\leq \gamma$.

  {\rm(ii)} There exists $P\in\myS_+^n(\myR)$, $\beta>1$ such that
\begin{eqnarray}
 \label{eq-ARI-1-our} && \frac{\beta^2}{\beta-1}B^TPB+D^TD\leq\gamma^2I_{n_v},\\
 \label{eq-ARI-2-our} && \beta(A^TPA+A_0^TPA_0)-P+C^TC\leq0.
\end{eqnarray}
\end{theorem}
\begin{proof}
  Firstly, we prove {\rm(i)} implies {\rm(ii)}. Since the system \eqref{eq-system-linear-1} is internally stable.
  By Lemma \ref{lem-stability-lemma-linear}, there exists $\bar{P}\in\myS_+^n(\myR)$. Taking $P=p_0\bar{P}$, let $\bar{V}(x)=x^T\bar{P}x$ and $V(x)=x^TPx$.
  Applying Assumption ($A''_2$) and Theorem \ref{th-bounded-real-lemma}, we prove that  $P$ satisfies \eqref{eq-ARI-1-our} and \eqref{eq-ARI-2-our}.

  As far as {\rm(ii)} implies {\rm(i)}, it   can be obtained directly by Proposition \ref{prop-stable-external} with $V(x)=x^TPx$, where $P\in\myS_+^n(\myR)$ satisfies \eqref{eq-ARI-1-our} and \eqref{eq-ARI-2-our}.
\end{proof}

\section{ $H_\infty$ control for general discrete-time stochastic systems}
In this section, we  consider the $H_\infty$ control of the
following general discrete-time stochastic system
\begin{eqnarray}\label{eq-system-control-general-F(x,u,v,omega)}
  \left\{\begin{array}{l}
    x_{k+1}=F_k(x_k,u_k,v_k,\omega_k),\\
    z_k=m_k(x_k,u_k,v_k)
  \end{array}\right.
\end{eqnarray}
where
$F_k:\myR^n\times\myR^{n_u}\times\myR^{n_v}\times\myR^d\rightarrow\myR^n$
and
$m_k:\myR^n\times\myR^{n_u}\times\myR^{n_v}\rightarrow\myR^{n_z}$
are measurable functions with
$F_k(0,\cdot,\cdot,\cdot)\equiv0,m_k(0,\cdot,\cdot)=0$,
$u(\cdot):=\{u_k\}_{k\in\myN}$ is the control sequence,
$v(\cdot):=\{v_k\}_{k\in\myN}$ is the exogenous disturbance sequence
with $v(\cdot)\in\myLiv$. Denote $\{x(k;s,x,u,v)\}_{k\in\myN}$ or
$\{x_k^{s,x,u,v}\}_{k\in\myN}$ the solution sequence of
\eqref{eq-system-control-general-F(x,u,v,omega)} with the initial
$x\in\myR^n$ starting at $k=s$ under the control $u(\cdot)$ and the
exogenous disturbance $v(\cdot)$, and the corresponding regulated
output  is denoted by $\{z(k;s,x,u,v)\}_{k\in\myN}$ or
$\{z_k^{s,x,u,v}\}_{k\in\myN}$. For each admissible control
$u(\cdot)$, define the operator $\myL_u$ by
$$
\myL_u[v(\cdot)]=z(\cdot,0,0,u,v), \ \  v(\cdot)\in\myLiv.
$$
The $H_\infty$ norm of $\myL_u$ is defined by
$$
\|\myL_u\|=\sup\limits_{v\in\myLiv}\frac{\|z(\cdot,0,0,u,v)\|_{l^2_\infty}}{\|v\|_{l^2_\infty}}
$$
We expect to find a state-feedback controller $u^*(\cdot)$ such that
the following closed-loop system of
\eqref{eq-system-control-general-F(x,u,v,omega)}
\begin{eqnarray}\label{eq-system-control-general-F(x,u,v,omega)44}
  \left\{\begin{array}{l}
    x_{k+1}=F_k(x_k,u^*_k(x_k),v_k,\omega_k),\\
    z_k=m_k(x_k,u^*_k(x_k),v_k)
  \end{array}\right.
\end{eqnarray}
 is externally
stable. Concretely speaking, for a given $\gamma>0$, find  a
state-feedback control sequence  $\{u^*_k=u^*_k(x_k)\}_{k\in\myN}$
such that $\|\myL_{u^*}\|\leq \gamma$, i.e.,
\begin{eqnarray*}
  \sum\limits_{k=0}^\infty \myE|z(k,0,0,u^*,v)|^2\leq \gamma^2\sum\limits_{k=0}^\infty \myE|v_k|^2, \ \ \forall v(\cdot)\in\myLiv.
\end{eqnarray*}
For a positive definite function sequence $\{V_k\}_{k\in \myN}$,
$V_k:\myR^n\rightarrow \myR^+$, and a positive real number
$\gamma>0$, we denote
\begin{eqnarray*}
  \Delta_{u,v}V_k(x):=\myE[V_{k+1}(F_k(x,u,v,\omega_k))]-V_k(x),\ \ u\in\myR^{n_u},v\in\myR^{n_v},k\in\myN,
\end{eqnarray*}
and
$$
H_k(x,u,v):=\Delta_{u,v}V_k(x)+|m_k(x,u,v)|^2.
$$
\begin{lemma}\label{lem-H-infty-general}
  Suppose, for given $\gamma>0$, there exist function sequences
  $\{\alpha_k\}_{k\in\myN}$  and   $\{V_k\}_{k\in\myN}$:
  $\alpha_k:\myR^n\rightarrow \myR^{n_u}$ and $V_k:\myR^n\rightarrow\myR^+$
  with $V_k(0)=0$, $k\in\myN$,  such that
  \begin{eqnarray}\label{eq-cond-H-control-general}
    H_k(x,\alpha_k(x),v)-\gamma^2|v|^2\leq 0,\quad \forall v\in\myR^{n_v}.
  \end{eqnarray}
  Then $u_k^*=\alpha_k(x_k)$ is the $H_\infty$ control of \eqref{eq-system-control-general-F(x,u,v,omega)}.
\end{lemma}
\begin{proof}
  Let $\{x_k\}_{k\in\myN}$ be the solution of \eqref{eq-system-control-general-F(x,u,v,omega)44} with control $u_k^*=\alpha_k(x_k)$ and initial $x_0=0$,
  and $z_k=m_k(x_k,\alpha_k(x_k),v_k)$ is the corresponding output, then
  \begin{eqnarray*}
    \myE[V_{k+1}(x_{k+1})]-\myE[V_k(x_k)]&=&\myE[V_{k+1}(F_k(x_{k},\alpha_k(x_k),v_k,\omega_k))]-\myE[V_k(x_k)]\\
    &=&\myE\big\{\myE[V_{k+1}(F_k(x_{k},\alpha_k(x_k),v_k,\omega_k))|\myF_k]\big\}\\
    &&\ \ \ -\myE[V_k(x_k)].
  \end{eqnarray*}
  Since $x_k$ and $v_k$ are  $\myF_k-$measurable and $\omega_k$ is independent of $\myF_k$,
  by Lemma \ref{lem-property-conditional-expectation}, we have
  \begin{eqnarray*}
    &&\myE[V_{k+1}(x_{k+1})]-\myE[V_k(x_k)]=\myE\big\{\myE[V_{k+1}(F_k(x,u,v,\omega_k))-V_k(x)|\myF_k]_{x=x_k,u=\alpha_k(x_k),v=v_k}\big\}\\
   && =\myE\big\{\big[\myE[V_{k+1}(F_k(x,u,v,\omega_k))]-V_k(x)\big]_{x=x_k,u=\alpha_k(x_k),v=v_k}\big\}\\
   && =\myE\{\big[\Delta_{u,v}V_k(x)]_{x=x_k,u=\alpha_k(x_k),v=v_k}\}\\
   && =\myE\big\{\big[\Delta_{u,v}V_k(x)+|m_k(x,u,v)|^2-\gamma^2|v|^2\big]_{x=x_k,u=\alpha_k(x_k),v=v_k}\big\} \\
   &&\quad-\myE[|z_k|^2]+\gamma^2\myE[|v_k|^2]\\
   && =\myE\big\{\big[H_k(x,u,v)-\gamma^2|v|^2\big]_{x=x_k,u=\alpha_k(x_k),v=v_k}\big\}-\myE[|z_k|^2]+\gamma^2\myE[|v_k|^2].
  \end{eqnarray*}
  Applying \eqref{eq-cond-H-control-general}, we have
    \begin{eqnarray*}
    \myE[V_{k+1}(x_{k+1})]-\myE[V_k(x_k)]\leq -\myE[|z_k|^2]+\gamma^2\myE[|v_k|^2].
  \end{eqnarray*}
  Taking summation  on both sides of the above inequality from $k=0$ to $N\in\myN$, we obtain
 \begin{eqnarray*}
    \myE[V_{N+1}(x_{N+1})]-\myE[V_0(0)]\leq
-\sum\limits_{k=0}^N\myE[|z_k|^2]+\gamma^2\sum\limits_{k=0}^N\myE[|v_k|^2].
  \end{eqnarray*}
  Keeping $V_0(0)=0$ and $V_k(x)>0$ for $k\in\myN$  and $x\ne 0$ in mind,  we have
  \begin{eqnarray*}
    \sum\limits_{k=0}^N\myE[|z_k|^2]\leq\gamma^2\sum\limits_{k=0}^N\myE[|v_k|^2].
  \end{eqnarray*}
  Let $N\rightarrow\infty$, we get
   \begin{eqnarray*}
    \sum\limits_{k=0}^\infty\myE[|z_k|^2]\leq\gamma^2\sum\limits_{k=0}^\infty\myE[|v_k|^2].
  \end{eqnarray*}
  This proves that $u^*_k=\alpha_k(x_k)$ is the $H_\infty$ control of system \eqref{eq-system-control-general-F(x,u,v,omega)}.
\end{proof}
\begin{theorem}
  For $\gamma>0$, suppose there exist  positive  functions $V_k:\myR^n\rightarrow \myR^+, k\in\myN$ with $V_k(0)=0$, which satisfy the following conditions:

  i) There exist $\alpha_k(x)$ and $\eta_k(x)$ such that for any $x\in\myR^n$ and
  $v\in\myR^{n_v}$,
  \begin{eqnarray*}
    \partial_uH_k(x,\alpha_k(x),\eta_k(x))=0,\ \ \partial_vH_k(x,\alpha_k(x),\eta_k(x))=0.
  \end{eqnarray*}

  ii) There exist matrices  $M\in\myS_+^{n_u}(\myR)$ and $N\in\myS^{n_v}(\myR)$, such that
  $$
  \left[\begin{array}{cc}
  \partial^2_{uu}H_k & \partial^2_{uv}H_k\\
  \partial^2_{vu}H_k & \partial^2_{vv}H_k
  \end{array}\right]
  \leq\left[\begin{array}{cc}
    M & 0\\
    0 & N
  \end{array}\right],\gamma^2I-N>0.
  $$

iii)
  $$\begin{array}{l}
     \myH(V_k(x)):=H_k(x,\alpha_k(x),\eta_k(x))+\frac{1}{2}\eta_k(x)^T\left[N+N(\gamma^2I_{n_v}-N)^{-1}N^T\right]\eta_k(x)\leq
0.
  \end{array}
  $$

  Then $\{u^*_k=\alpha_k(x)\}_{k\in\myN}$ is the $H_\infty$ controller for system \eqref{eq-system-control-general-F(x,u,v,omega)}
\end{theorem}
\begin{proof}
By Taylor's series expansion, it follows that
 \begin{eqnarray*}
   &&H_k(x,u,v)=H_k(x,\alpha_k(x),\eta_k(x))\\
  && \quad\quad\quad\quad\quad\quad+\langle\partial_uH_k(x,\alpha_k(x),\eta_k(x)),u-\alpha_k(x)\rangle\\
  && \quad\quad\quad\quad\quad\quad+\langle\partial_vH_k(x,\alpha_k(x),\eta_k(x)),v-\eta_k(x)\rangle\\
  && \quad\quad\quad\quad\quad\quad+\frac{1}{2}\bigg[(u-\alpha_k(x))^T\partial^2_{uu}H_k(\bar{\theta})(u-\alpha_k(x))\\
  && \quad\quad\quad\quad\quad\quad+2(u-\alpha_k(x))^T\partial^2_{uv}H_k(\bar{\theta})(v-\eta_k(x))\\
  &&\quad\quad\quad\quad\quad\quad +(v-\eta_k(x))^T\partial^2_{vv}H_k(\bar{\theta})(v-\eta_k(x))\bigg]
 \end{eqnarray*}
 where $\bar{\theta}=(x,\alpha_k(x)+\theta(u-\alpha_k(x)),\eta_k(x)+\theta(v-\eta_k(x)))$, $0<\theta<1$. So
 \begin{eqnarray*}
   H_k(x,u,v)-\gamma^2|v|^2&\leq& H_k(x,\alpha_k(x),\eta_k(x))+\frac{1}{2}\bigg[(u-\alpha_k(x))^TM(u-\alpha_k(x))\\
   &&\quad\quad\quad\quad\quad\quad+(v-\eta_k(x))^TN(v-\eta_k(x))\bigg]-\gamma^2|v|^2.
 \end{eqnarray*}
Completing  squares with respect to  $v$ on the right hand side of
the  above inequality, we have
 \begin{eqnarray*}
   H_k(x,u,v)-\gamma^2|v|^2 &\leq&  \myH(V_k(x))  +\frac{1}{2}\|(u-\alpha_k(x))\|^2_{M}\\
   &&-\|v+(\gamma^2I_{n_v}-N)^{-1}\eta_k(x)\|^2_{\gamma^2I_{n_v}-\frac{1}{2}N}.
 \end{eqnarray*}
 Applying condition  iii), we obtain
  \begin{eqnarray*}
   H_k(x,u,v)-\gamma^2|v|^2\leq  \frac{1}{2}\|(u-\alpha_k(x))\|^2_{M}.
 \end{eqnarray*}
 So, for $u^*=\alpha_k(x)$, there is
   \begin{eqnarray*}
   H_k(x,\alpha_k(x),v)-\gamma^2|v|^2\leq 0.
 \end{eqnarray*}
 By Lemma \ref{lem-H-infty-general},  $u^*_k=\alpha_k(x)$, $k\in\myN$, are  the $H_\infty$ control for system \eqref{eq-system-control-general-F(x,u,v,omega)}.
\end{proof}

Now, we consider the special time-invariant case with affine form of
\eqref{eq-system-control-H-infty}. Denote
\begin{eqnarray}
\nonumber\myH(V(x),u,\beta):=\frac{1}{\beta}\myE[V(\beta
f(x,u,\omega_0))]-V(x)+|m(x,u)|^2
\end{eqnarray}
\begin{theorem}\label{th-Hinfty-control-design}
  For $\gamma>0$, if there exist a positive convex function $V:\myR^n\rightarrow \myR^+$, a positive real number $\beta>1$ and a  function
  $\alpha:\myR^n\rightarrow\myR^{n_u}$,  such that
  \begin{eqnarray}
   \label{eq-H-Hinfty-cond-1} &&\myH(V(x),\alpha(x),\beta)\leq0,\\
   \label{eq-H-Hinfty-cond-2} &&G_\beta(V(x))\leq \gamma^2,\forall x\in\myR^n
  \end{eqnarray}
  where $G_\beta$ is defined by \eqref{eq-G-beta}, then $u^*_k=\alpha(x_k)$, $k\in\myN$,  is the $H_\infty$ control of system \eqref{eq-system-control-H-infty} and $\|\myL_{u^*}\|\leq\gamma$.
\end{theorem}

\begin{proof}
 Substituting  $u^*_k=\alpha(x_k)$ into system \eqref{eq-system-control-H-infty}, it is easy to see that the inequalities \eqref{eq-H-Hinfty-cond-1} and
 \eqref{eq-H-Hinfty-cond-2} are  same with   \eqref{eq-H-1-ineq-1} and \eqref{eq-H-1-ineq-2}, respectively. By Proposition \ref{prop-stable-external}, we know that system \eqref{eq-system-control-H-infty} under control $u^*$ is externally stable with $\|\myL_{u^*}\|\leq \gamma$.
\end{proof}
\begin{remark}
How to solve the  inequality \eqref{eq-H-Hinfty-cond-1} is the key
to design the $H_\infty$ controller. If, for every $x\in\myR^n$,
there exists $\alpha^*(x)$ such that
  $$
  \alpha^*(x)=\arg\min\limits_{u\in\myR^{n_v}}\myH(V(x),u,\beta),
  $$
then under   the  condition  \eqref{eq-H-Hinfty-cond-1},
  $$
  \myH(V(x),\alpha^*(x),\beta)\leq\myH(V(x),\alpha(x),\beta)\leq 0.
  $$
  So, we can choose $\alpha^*(x)$ as the $H_\infty$ controller for system \eqref{eq-system-control-H-infty}.
\end{remark}

\section{Simulation Examples}
In this section, we present some numerical examples to illustrate
the effectiveness of our developed theory.
\begin{example}
Consider the following 1-dimensional system
\begin{eqnarray}\label{eq-example-1-system}
\left\{\begin{array}{l}
    x_{k+1}=ax_k+b\cos(x_k)\omega_kv_k,\\
  z_k=\left[\begin{array}{c}
  cx_k\\
  c_1|v_k|
  \end{array}\right],
  \end{array}\right.
\end{eqnarray}
where $a,b,c$ and $c_1$ are real numbers with $|a|<1$. Take
$$
\myfunV=\{V:V(x)=px^2,x\in\myR, p>0\}.
$$
Then
\begin{eqnarray}
 \label{eq-example-1-H-1}&& \myH_1(V(x),\beta)=pa^2\beta x^2- px^2+c^2x^2,\\
\nonumber && G_\beta(V(x))=\frac{p\beta b^2\cos^2x}{\beta-1}+c_1^2,
\beta>1.
\end{eqnarray}
By \eqref{eq-H-1-ineq-1} and \eqref{eq-example-1-H-1}, we have
\begin{eqnarray*}
 \nonumber pa^2\beta-p+c^2\leq 0.
\end{eqnarray*}
In order that the  above inequality is solvable, it only needs the
following two inequalities to be held:
\begin{eqnarray*}
1<\beta\leq\frac{1}{a^2}, \ \  p\geq\frac{c^2}{1-a^2\beta}.
\end{eqnarray*}
Since
\begin{eqnarray*}
  \sup\limits_{x\in\myR}\{G_\beta(V(x))\}=\frac{p\beta
  b^2}{\beta-1}+|c_1|^2,
\end{eqnarray*}
if we take $\beta=\frac{1}{|a|}$, $p=\frac {c^2}{1-a^2\beta}$,  then
$$
\sup\limits_{x\in\myR}\{G_{\frac{1}{|a|}}(V(x))\}=\frac{b^2c^2}{(1-|a|)^2}+|c_1|^2.
$$
Set
$$
\gamma^{*^2}=\frac{b^2c^2}{(1-|a|)^2}+|c_1|^2.
$$
 So, by Proposition~\ref{prop-stable-external},  system \eqref{eq-example-1-system} is
 not only
 externally stable with $\|\myL\|\leq\gamma^*$,  but also internally stable.
Fig.1 shows the simulations of the trajectories of
$|z_k|^2,\gamma^{*^2}v_k^2$ and $v_k^2$ of system
\eqref{eq-example-1-system} with coefficients
$a=0.99,b=0.01,c=c_1=0.2$ and the initial state $x_0=0$. From the
simulations,  we can  see that the inequality
\eqref{eq-definition-external-z-norm} holds and
$\|\myL\|\leq\gamma^*$.
\begin{figure}[htb]\label{fig-exampel-1-1}
\center
  \includegraphics[width=8cm,height=4.8cm]{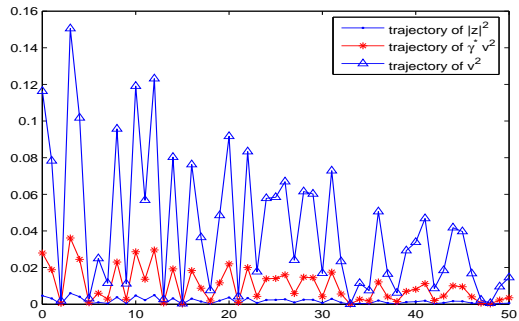}
  \caption{Trajectories of $|z|^2,\gamma^{*^2}v^2$ and disturbance $v^2$ of system
\eqref{eq-example-1-system}}
\end{figure}
\end{example}

\begin{example}
Considering the following stochastic system
  \begin{eqnarray}\left\{
  \begin{array}{l}
      x^{(1)}_{k+1}=\theta^{(1)}_kx^{(1)}_k+\theta^{(2)}_k|x^{(2)}_k|^2+u^{(1)}_k+v^{(1)}_k,\\
    x^{(2)}_{k+1}=\theta^{(3)}_kx^{(2)}_k+\theta^{(4)}_k\frac{x^{(3)}_k}{1+|x^{(3)}_k|}+u^{(2)}_k,\\
    x^{(3)}_{k+1}=\theta^{(5)}_kx^{(3)}_k\cos(x^{(2)}_k)
    +u^{(1)}_k+v^{(2)}_k,\\
    x^{(1)}_0\in\myR,  x^{(2)}_0\in\myR, x^{(3)}_0\in(0,1), k\in\myN
    \end{array}\right.\label{eq-example-2-system}
  \end{eqnarray}
  with the controlled output
  \begin{equation}\label{eq-example-2-output}
    z_k=\left[\begin{array}{c}
      0.1x^{(1)}_k+0.1x^{(3)}_k\cos(x^{(2)}_k)\\
      \frac{1}{7}|x^{(2)}_k|^2\\
      u^{(1)}_k
    \end{array}\right],
  \end{equation}
  where $\{\theta^{(i)}_k,i=1,2,3,4,5\}_{k\in\myN}$ are independently
  identically distributed
  random variable sequences, and $\theta^{(1)}_k,\cdots, \theta^{(5)}_k$ are also
  independent of each other. Moreover,
  $\theta^{(1)}_k$, $\theta^{(3)}_k$, $\theta^{(4)}_k$ and $\theta^{(5)}_k$
  are uniformly distributed on $[0,1]$, and  $\theta^{(2)}_k$ is uniformly distributed
  on $[-1/2,1/2]$. $\{u^{(1)}_k,u^{(2)}_k\}_{k\in\myN}$  are the control sequences and
  $\{v^{(1)}_k\}_{k\in\myN}$
  is the  exogenous disturbance  sequence.

  Denote $\omega_k=(\theta^{(1)}_k,\theta^{(2)}_k,\theta^{(3)}_k,\theta^{(4)}_k,\theta^{(5)}_k)^T$,
  $x=(x^{(1)},x^{(2)},x^{(3)})^T$, $u=(u^{(1)},u^{(2)})^T$ and $v=(v^{(1)},v^{(2)})^T$, then
  the corresponding  $f$, $g$, $m$ and $m_1$ in \eqref{eq-system-control-H-infty} can
  be written as
$$
f(x,u,\omega_k)=
  \left[\begin{array}{c}
    \theta^{(1)}_kx^{(1)}+\theta^{(2)}_k(x^{(2)})^2+u^{(1)}\\
    \theta^{(3)}_kx^{(2)}+\theta^{(4)}_k\frac{x^{(3)}}{1+|x^{(3)}|}+u^{(2)}\\
    \theta^{(5)}_kx^{(3)}\cos(x^{(2)})+u^{(1)}
  \end{array}\right], \ \ g(x,\omega_k)=\left[\begin{array}{cc}
    1 & 0\\
    0 & 0\\
    0 & 1
  \end{array}\right],
$$
$$
m(x,u)=\left[\begin{array}{c}
      0.1x^{(1)}+0.1x^{(3)}\cos(x^{(2)})\\
      \frac{1}{7}(x^{(2)})^2\\
      u^{(1)}
    \end{array}\right], \ \
     m_1(x)=0.
$$
  Suppose the function $V:\myR^3\rightarrow\myR^+$ has the form of
  \begin{equation}
    V(x)=p_1(x^{(1)})^2+p_2(x^{(2)})^4+p_3(x^{(3)})^2.
  \end{equation}
  For each  $\beta>1$, we have
  \begin{eqnarray*}
\myH(V(x),u,\beta)&=&\beta\myE\bigg[p_1(\theta^{(1)}_kx^{(1)}+\theta^{(2)}_k(x^{(2)})^2+u^{(1)})^2+p_2(\theta^{(3)}_kx^2+\theta^{(4)}_k\frac{x^{(3)}}{1+|x^{(3)}|}+u^{(2)})^4\\
    &&+p_3(\theta^{(5)}_k\cos(x^{(2)})x^{(3)}+u^{(1)})^2\bigg]-[p_1(x^{(1)})^2+p_2(x^{(2)})^4+p_3(x^{(3)})^2]\\
   && +0.01[x^{(1)}+x^{(3)}\cos(x^{(2)})]^2+\frac{1}{49}(x^{(2)})^4+(u^{(1)})^2
  \end{eqnarray*}
  and
  \begin{eqnarray*}
G_\beta(V(x))=\frac{\beta}{\beta-1}\sup\limits_{0\neq
v\in\myR^2}\frac{p_1(v^{(1)})^2+p_3(v^{(2)})^2}{|v|^2}=\frac{\beta}{\beta-1}\max(p_1,p_3).
  \end{eqnarray*}
   For $\gamma=0.75$, taking
   $$
   \beta=\sqrt[3]{8/5},p_1=p_2=p_3=p=\frac{1}{16},
   $$
   $$
   u^{(1)*}=-\frac{\beta^3p}{4\beta^3p+2}[x^{(1)}+x^{(3)}\cos(x^{(2)})]
   $$
   and $u^{(2)*}=-\frac{1}{2}[x^{(2)}+\frac{x^{(3)}}{1+|x^{(3)}|}]$, we have
\begin{eqnarray*}
    \myH(V(x),u^*,\beta)&\leq& p\beta^3\myE\bigg[(\theta^{(1)}_kx^{(1)}+\theta^{(2)}_k (x^{(2)})^2+u^{(1)*})^2+(\theta^{(3)}_kx^{(2)}+\theta^{(4)}_k\frac{x^{(3)}}{1+|x^{(3)}|}+u^{(2)*})^4\\
    &&+(\theta^{(5)}_k\cos(x^{(2)})x^{(3)}+u^{(1)*})^2\bigg]-p[(x^{(1)})^2+(x^{(2)})^4+(x^{(3)})^2]\\
    &&+0.01[x^{(1)}+x^{(3)}\cos(x^{(2)})]^2+\frac{1}{49}(x^{(2)})^4+|u^{(1)*}|^2\\
   &=&(2p\beta^3+1)\bigg[u^{(1)*}+\frac{p\beta^3}{2(2p\beta^3+1)}(x^{(1)}+x^{(3)}\cos(x^{(2)}))\bigg]^2\\
   &&+p\beta^3\bigg[\theta^{(3)}_kx^{(2)}+\theta^{(4)}_k\frac{x^{(3)}}{1+|x^{(3)}|}+u^{(2)*}\bigg]^4\\
   &&+p\beta^3\bigg[\frac{1}{3}(x^{(1)})^2+\frac{1}{12}(x^{(2)})^4+\frac{1}{3}|\cos(x^{(2)})|^2(x^{(3)})^2\bigg]\\
   &&-\frac{p^2\beta^6}{4(2p\beta^3+1)}[x^{(1)}+x^{(3)}\cos(x^{(2)})]^2-p(x^{(1)})^2\\
   &&-p(x^{(2)})^4-p(x^{(3)})^2+0.01[x^{(1)}+x^{(3)}\cos(x^{(2)})]^2+\frac{1}{49}(x^{(2)})^4
\end{eqnarray*}
\begin{eqnarray*}
   &&=p\beta^3\bigg[\frac{1}{80}(x^{(2)})^4+\frac{1}{24}\frac{(x^{(2)})^2(x^{(3)})^2}{(1+|x^{(3)}|)^2}\\
   &&\quad+\frac{1}{80}\frac{(x^{(3)})^4}{(1+|x^{(3)}|)^4}+\frac{1}{3}(x^{(1)})^2+\frac{1}{12}(x^{(2)})^4\\
   &&\quad+\frac{1}{3}|\cos(x^{(2)})|^2(x^{(3)})^2\bigg]-\frac{p^2\beta^6}{4(2p\beta^3+1)}[x^{(1)}+x^{(3)}\cos(x^{(2)})]^2\\
&&\quad-p(x^{(1)})^2-p(x^{(2)})^4-p(x^{(3)})^2+0.01[x^{(1)}+x^{(3)}\cos(x^{(2)})]^2+\frac{1}{49}(x^{(2)})^4\\
   &&\leq p\beta^3\bigg[\frac{1}{80}(x^{(2)})^4+\frac{1}{24}\frac{(x^{(2)})^2(x^{(3)})^2}{(1+|x^{(3)}|)^2}+\frac{1}{80}\frac{(x^{(3)})^4}{(1+|x^{(3)}|)^4}\\
   &&\quad+\frac{1}{3}(x^{(1)})^2+\frac{1}{12}(x^{(2)})^4+\frac{1}{3}|\cos(x^{(2)})|^2(x^{(3)})^2\bigg]\\
   &&\quad-p(x^{(1)})^2-p(x^{(2)})^4-p(x^{(3)})^2+0.01[x^{(1)}+x^{(3)}\cos(x^{(2)})]^2+\frac{1}{49}(x^{(2)})^4\\
   &&\leq p\beta^3\bigg[\frac{1}{80}(x^{(2)})^4+\frac{1}{24}\frac{(x^{(2)})^2(x^{(3)})^2}{(1+|x^{(3)}|)^2}+\frac{1}{80}\frac{(x^{(3)})^4}{(1+|x^{(3)}|)^4}\\
   &&\quad+\frac{1}{3}(x^{(1)})^2+\frac{1}{12}(x^{(2)})^4+\frac{1}{3}(x^{(3)})^2\bigg]\\
   &&\quad-p(x^{(1)})^2-p(x^{(2)})^4-p(x^{(3)})^2+0.02[(x^{(1)})^2+(x^{(3)})^2]+\frac{1}{49}(x^{(2)})^4\\
   &&\leq p\beta^3\bigg[\frac{1}{80}(x^{(2)})^4+\frac{1}{48}((x^{(2)})^4+(x^{(3)})^2)+\frac{1}{80}(x^{(3)})^2\\
   &&\quad+\frac{1}{3}(x^{(1)})^2+\frac{1}{12}(x^{(2)})^4+\frac{1}{3}(x^{(3)})^2\bigg]\\
   &&\quad-p(x^{(1)})^2-p(x^{(2)})^4-p(x^{(3)})^2+0.02[(x^{(1)})^2+(x^{(3)})^2]+\frac{1}{49}(x^{(2)})^4\\
   &&=[\frac{p\beta^3}{3}+0.02-p](x^{(1)})^2+[p\beta^3(\frac{1}{12}+\frac{1}{80}+\frac{1}{48})+\frac{1}{49}-p](x^{(2)})^4\\
   &&\quad+[p\beta^3(\frac{1}{3}+\frac{1}{48}+\frac{1}{80})+0.02-p](x^{(3)})^2\\
   &&\leq\left[\frac{5}{12}p\beta^3+\frac{1}{48}-p\right]\left[(x^{(1)})^2+(x^{(2)})^4+(x^{(3)})^2\right]\\
   &&=0.
  \end{eqnarray*}
  As far as
  \begin{eqnarray*}
    G_\beta(V(x))=\frac{\beta}{\beta-1}\frac{1}{16}\leq\gamma^2,
  \end{eqnarray*}
it can be obtained by the fact that
  \begin{eqnarray*}
    5\times8^3\beta^3=8^4=4096>3645=5\times9^3.
  \end{eqnarray*}
So, for any $x\in\myR^3$,
   the above given $V(x)$, $u^*=(u^{(1)*},u^{(2)*})^T$ and $\beta$ satisfy conditions of
   \eqref{eq-H-Hinfty-cond-1} and \eqref{eq-H-Hinfty-cond-2}.
   According to Theorem~\ref{th-Hinfty-control-design}, $u^*=(u^{(1)*},
   u^{(2)*})^T$ is the corresponding $H_\infty$ control of system
   (\ref{eq-example-2-system}). Moreover, system
   (\ref{eq-example-2-system}) is internally stable under the
   $H_\infty$ control $u^*$.
   Fig.2  shows the trajectories of $H_\infty$ control
   $u^{(1)*}$ and  $u^{(2)*}$. Fig.3  shows the samples of the trajectories of the
   states $x^{(1)}$, $x^{(2)}$ and $x^{(3)}$ under the control $u^*$.  Fig.4 shows the trajectories
   of $\gamma^2 |v|^2$ and  $|z|^2$. From Fig.4,
   we see that $\|\myL_{u^*}\|\leq \gamma$,  which verifies the
   correctness of
   Theorem~\ref{th-Hinfty-control-design}.
   \begin{figure}\label{fig-exampel-2-1}
   \begin{center}
     \includegraphics[width=8cm,height=4.8cm]{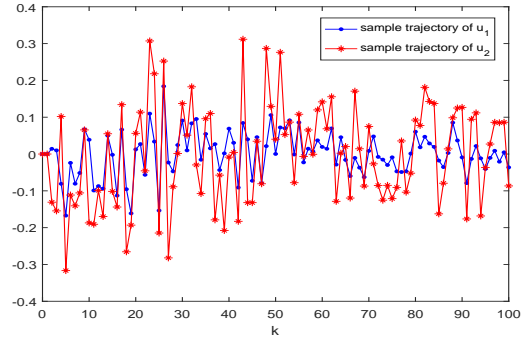}
     \caption{Trajectories of $H_\infty$ control $u^*$ for system \eqref{eq-example-2-system}}
     \end{center}
   \end{figure}
   \begin{figure}\label{fig-exampel-2-2}
   \begin{center}
     \includegraphics[width=8cm,height=4.8cm]{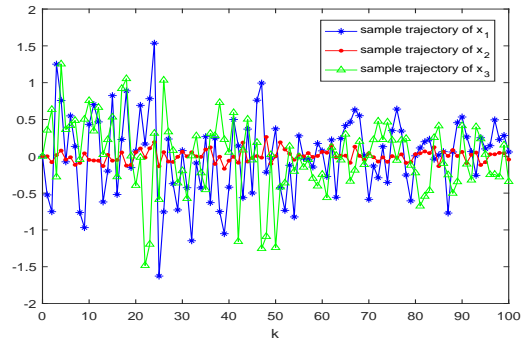}
     \caption{Trajectories of system \eqref{eq-example-2-system} under $H_\infty$ control}
     \end{center}
   \end{figure}
   \begin{figure}\label{fig-exampel-2-3}
   \begin{center}
     \includegraphics[width=8cm,height=4.8cm]{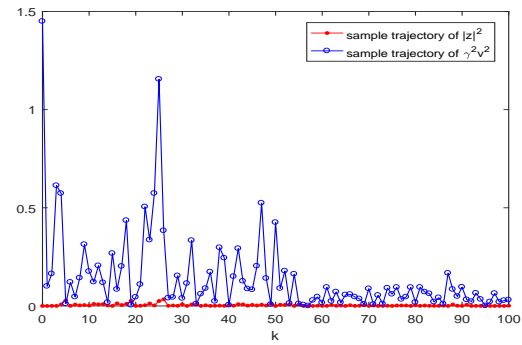}
     \caption{Trajectories of $|z|^2$ and $\gamma^2v^{(2)}$ of
     system \eqref{eq-example-2-system} under $H_\infty$ control}
     \end{center}
   \end{figure}
\end{example}

\section{Conclusions}
 We have introduced the convex analysis  method   to study the $H_\infty$ control
 for  more general discrete-time nonlinear stochastic systems (see systems (\ref{eq-system-control-H-infty}) and
 (\ref{eq-system-control-general-F(x,u,v,omega)})),
based on which,  a stochastic version of bounded real lemma for
discrete-time nonlinear stochastic systems has been obtained. It can
be found that our concerned systems are more general than affine
nonlinear system (\ref{eq iih13_2.2.3dfd}). It is expected that the
developed   convex analysis technique  can also be applied to deal
with the output feedback $H_\infty$ control and  other robust
control problems  such as in \cite{LinByrnes1995}.


\end{document}